\documentclass[11pt]{amsart}

\usepackage{amsmath}
\usepackage{amsthm, amssymb}
\usepackage{mathtools}
\usepackage{tkz-graph}
\usepackage{array}
\usepackage{enumitem}

\newtheorem{theorem}{Theorem}
\newtheorem{proposition}{Proposition}
\newtheorem{lemma}{Lemma}
\newtheorem{conjecture}{Conjecture}

\newtheorem{definition}{Definition}
\newtheorem{question}{Question}

\title[Making multigraphs simple]{Making multigraphs simple by a sequence of double edge swaps}

\date{September 2021}
\author{Jonas Sj\"{o}strand}
\email{jonas.sjostrand@mdh.se}
\address{M\"alardalen University, School of Education, Culture and Communication, Box~883, SE-72123 V\"aster\aa s, Sweden}
\begin{document}

\begin{abstract}
We show that any loopy multigraph with a graphical degree sequence
can be transformed into a simple graph by a finite sequence of 
double edge swaps with each swap involving at least one loop or multiple edge.
Our result answers a question of Janson motivated by random graph theory,
and it adds to the rich literature on reachability of double edge swaps
with applications in Markov chain Monte Carlo sampling from the
uniform distribution of graphs with prescribed degrees.
\end{abstract}

\maketitle

\section{Introduction}
We will consider different classes of undirected graphs, the most general
being \emph{loopy multigraphs} where both multiple edges and multiple
loops are allowed. Specifically, we are interested in graphs
where each vertex has a prescribed \emph{degree}, the degree of a vertex
being the number of stubs (half-edges) attached to it (so the contribution from a loop is two). The list of the degrees of
all vertices, sorted in weakly decreasing order, is called the
\emph{degree sequence} of the graph, and
a weakly decreasing sequence is
said to be \emph{graphical} if it is the degree sequence of some
simple graph (no loops and no multiple edges).
The most popular basic graph operation that preserves the degree sequence is
the replacement of any two edges $(v_1,v_2)$ and $(v_3,v_4)$ by
$(v_2,v_3)$ and $(v_4,v_1)$. This is called a \emph{double edge swap}
and was first introduced by Petersen~\cite{petersen1891}. It
has been reinvented several times and has many alternative names in the
literature \cite{Fosdick2018}: 
degree-preserving rewiring, checkerboard swap,
tetrad or alternating rectangle.

The main motivation for our work comes from the theory of random graphs.
There is a simple direct method of generating
a uniformly random stub-labelled (where the stubs have identity)
loopy multigraph with prescribed degrees: Attach the prescribed
number of stubs to each vertex, then choose a random
matching of all stubs. This is called the \emph{configuration model}
and was introduced by Bollob\'as 1980~\cite{bollobas1980}.
The simplicity of the method makes it very useful
for theoretical analyses of random graphs, but in
many applications one wants to study simple graphs
rather than multigraphs.
There are several possible solutions to this issue.
Sometimes it is possible to simply condition the random loopy multigraph
from the configuration model on the event that it is a simple graph.
This yields a uniform distribution of simple graphs with the
given degree sequence. Recently, Janson~\cite{Janson2019}
proposed another method, the \emph{switched configuration model},
where the random loopy multigraph is transformed into a simple graph
by a sequence of random double edge swaps.
Each swap is required to
have the property that at least one of the two swapped edges
is a loop or a multiple edge.
The resulting distribution on simple
graphs is not exactly uniform, but for a certain class of degree sequences
Janson showed that it is asymptotically uniform in the sense that the
total variation distance to the uniform distribution tends to zero
when the number of vertices goes to infinity. Motivated by his construction,
he posed the following question to us in person:
\begin{question}\label{qu:janson}
Can any loopy multigraph with a graphical degree sequence
be transformed into a simple graph by a finite sequence of 
double edge swaps involving at least one loop or multiple edge?
\end{question}
In this paper, we answer the question affirmatively. In fact,
we show a stronger statement that Jansson conjectured in
\cite[Remark~3.4]{Janson2019}, namely that it is always possible
to reach a simple graph even if an evil person chooses which
loop or multiple edge should be involved in each double edge swap.

Our result adheres to a rich literature of reachability of double
edge swaps, a topic that has an important application in the context of
Markov chain Monte Carlo sampling; see Fosdick et al.~\cite{Fosdick2018}
for a comprehensive discussion.
In the simplest case, we want to sample from the uniform distribution
of all graphs (of some class)
with prescribed degrees. Basically, one starts with any
graph with the given degrees and performs random
double edge swaps for a while;
the stationary distribution is
uniform. (Exactly how the random double edge swaps should be
chosen depends on the class of graphs and the type of labelling
of the graph, see \cite{Fosdick2018}.)
To show uniformity, one has to verify that the Markov chain satisfies
three conditions:
\begin{enumerate}
\item[(i)] that the transition matrix of the chain is doubly stochastic,
\item[(ii)] that the chain is irreducible,
\item[(iii)] and that the chain is aperiodic.
\end{enumerate}
The irreducibility condition means that
for any pair of graphs $G$ and $G'$ with the same degree sequence
there is a sequence of double edge swaps that transforms $G$ to $G'$.
If this is true or not depends on the particular class of graphs we are
interested in.
It is true for simple graphs \cite{Eggleton1975, BienstockGunluk1994, EggletonHolton1981}, connected simple graphs \cite{Taylor1981},
2-connected simple graphs \cite{Taylor1982},
loop-free multigraphs \cite{Hakimi1963}, simple-loopy multigraphs
(multiple edges and simple loops) \cite{Nishimura2017} and loopy multigraphs \cite{EggletonHolton1979},
but not for simple-loopy simple graphs
(simple edges and simple loops) \cite{Nishimura2018}
and loopy simple graphs (where
multiple loops are allowed but no other multiple edges)
\cite{Nishimura2017}.

Note how our result differs from that of Eggleton and Holton
\cite{EggletonHolton1979}. While they show that any loopy multigraph
can be transformed into any other loopy multigraph with the same degree sequence
by a sequence of double edge swaps, we show that this can be accomplished with
\emph{admissible} swaps only, where a swap is admissible if it involves
at least one loop or multiple edge. In the situtation where Janson posed
Question~\ref{qu:janson}, this condition is natural since the goal is to
reach a simple graph. In applications, when using the switched configuration model
to sample from an
approximately uniform distribution of simple graphs with a given degree sequence,
one wants to obtain a simple
graph by as few double edge swaps as possible, so swapping away ``bad'' edges is essential for the efficiency of this method.

The paper is organized as follows. First, in
Section~\ref{sec:notation} we fix the notation
and recall the Erd\H{o}s-Gallai theorem. In Section~\ref{sec:results} we present our results
and in Sections~\ref{sec:shortproof} and~\ref{sec:altproof} we prove them.
Finally, in Section~\ref{sec:open} we discuss
some open questions.

\section{Notation and prerequisites\label{sec:notation}}
The terminology on multigraphs is not standardized,
so let us start by defining it. Figure~\ref{fig:graphexamples}
shows some examples.

A \emph{loop} is an edge connecting a vertex to itself.
A \emph{loopy multigraph} is an undirected graph where loops are allowed
and where there might be multiple edges between the same pair
of vertices and multiple loops at the same vertex.

A \emph{loop-free multigraph} is a loopy multigraph without loops.

An edge is said to be \emph{simple} if it has multiplicity one and is not
a loop,
and a graph is \emph{simple} if all its edges are simple.

The \emph{degree} of a vertex is the number of half-edges adjacent to it
(so each loop contributes with two to the degree).
The list of the degrees of
all vertices, sorted in weakly decreasing order, is called the
\emph{degree sequence} of the graph.

We will denote an edge between $v_1$ and $v_2$
with curly braces $\{v_1,v_2\}$ and sometimes, to stress the
difference between an edge and an unordered pair of vertices,
we will talk about an edge of \emph{type} $\{v_1,v_2\}$.

\begin{figure}
\begin{center}
\begin{tabular}{ccc}
\begin{tikzpicture}[scale=0.8]
\GraphInit[vstyle=Classic]
\tikzset{VertexStyle/.append style = {minimum size = 2 mm}}
\Vertex[L=8, Lpos=90, x=0, y=0]{u}
\Vertex[L=5, Lpos=180, x=-1.5, y=1]{v}
\Vertex[L=3, Lpos=180, x=-1.5, y=-1]{w}
\Edge[style={bend left, out=20, in=160}](u)(v)
\Edge[style={bend right, out=-20, in=-160}](u)(v)
\Edge[style={bend left, out=20, in=160}](v)(w)
\Edge[style={bend right, out=-20, in=-160}](v)(w)
\Edge(u)(v)
\Edge(w)(u)
\Loop[dist=1cm,dir=EA,style={-, out=0}](u)
\Loop[dist=1cm,dir=SO,style={-, out=-90}](u)
\end{tikzpicture}
&
\begin{tikzpicture}[scale=0.8]
\GraphInit[vstyle=Classic]
\tikzset{VertexStyle/.append style = {minimum size = 2 mm}}
\Vertex[L=4, Lpos=90, x=0, y=0]{u}
\Vertex[L=5, Lpos=180, x=-1.5, y=1]{v}
\Vertex[L=3, Lpos=180, x=-1.5, y=-1]{w}
\Edge[style={bend left, out=20, in=160}](u)(v)
\Edge[style={bend right, out=-20, in=-160}](u)(v)
\Edge[style={bend left, out=20, in=160}](v)(w)
\Edge[style={bend right, out=-20, in=-160}](v)(w)
\Edge(u)(v)
\Edge(w)(u)
\end{tikzpicture}
&
\begin{tikzpicture}[scale=0.8]
\GraphInit[vstyle=Classic]
\tikzset{VertexStyle/.append style = {minimum size = 2 mm}}
\Vertex[L=2, Lpos=90, x=0, y=0]{u}
\Vertex[L=1, Lpos=180, x=-1.5, y=1]{v}
\Vertex[L=1, Lpos=180, x=-1.5, y=-1]{w}
\Edge(u)(v)
\Edge(w)(u)
\end{tikzpicture}
\\
a loopy multigraph & a loop-free multigraph & a simple graph
\end{tabular}
\end{center}
\caption{Examples of graphs and the degrees of their vertices.}
\label{fig:graphexamples}
\end{figure}
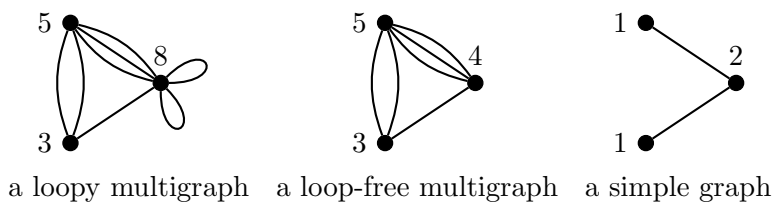

Two edges are said to be \emph{incident} if they share at least one vertex.
\begin{definition}
Suppose there are two edges of types
$\{v_1,v_2\}$ and $\{v_3,v_4\}$. Then we define
the \emph{double edge swap} $(v_1,v_2)(v_3,v_4)$
as the operation of removing two such edges and
adding two edges of type $\{v_2,v_3\}$ and $\{v_4,v_1\}$.

The swap is \emph{admissible} if the edges
$\{v_1,v_2\}$ and $\{v_3,v_4\}$
are not incident and not both of them are simple (before the swap).
\end{definition}
\noindent
See Figure~\ref{fig:swaps} for an illustration.
Clearly, a double edge swap (admissible or not) leaves the
degree sequence unchanged. Note also that an admissible double
edge swap never introduces a new loop since the edges are not incident.

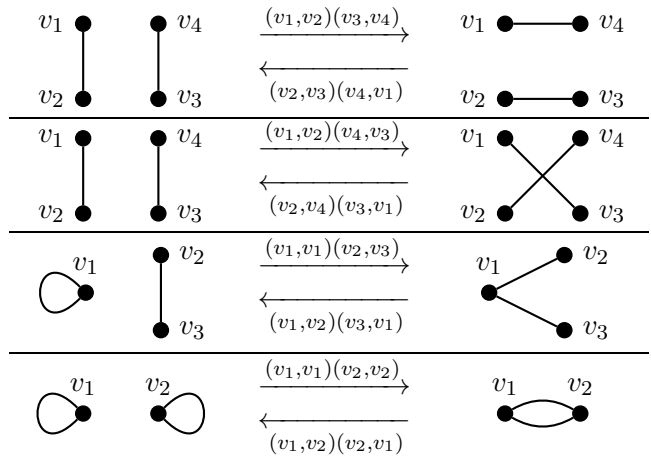
\begin{figure}
\begin{center}
\begin{tabular}{ccc}
\begin{tikzpicture}[baseline=0]
%[baseline=(current bounding box.center)]
\GraphInit[vstyle=Classic]
\SetGraphUnit{1}
\tikzset{VertexStyle/.append style = {minimum size = 2 mm}}
\Vertex[Math, Lpos=180, x=0, y=0.5]{v_1}
\SO[Math, Lpos=180](v_1){v_2}
\EA[Math](v_1){v_4}
\SO[Math](v_4){v_3}
\Edge(v_1)(v_2)
\Edge(v_3)(v_4)
\end{tikzpicture}
&
$
\begin{array}{c}
\xrightarrow{(v_1,v_2)(v_3,v_4)} \\
\xleftarrow[(v_2,v_3)(v_4,v_1)]{}
\end{array}
$
&
\begin{tikzpicture}[baseline=0]
%[baseline=(current bounding box.center)]
\GraphInit[vstyle=Classic]
\SetGraphUnit{1}
\tikzset{VertexStyle/.append style = {minimum size = 2 mm}}
\Vertex[Math, Lpos=180, x=0, y=0.5]{v_1}
\SO[Math, Lpos=180](v_1){v_2}
\EA[Math](v_1){v_4}
\SO[Math](v_4){v_3}
\Edge(v_1)(v_4)
\Edge(v_2)(v_3)
\end{tikzpicture}
\\
\hline
\begin{tikzpicture}[baseline=0]
%[baseline=(current bounding box.center)]
\GraphInit[vstyle=Classic]
\SetGraphUnit{1}
\tikzset{VertexStyle/.append style = {minimum size = 2 mm}}
\Vertex[Math, Lpos=180, x=0, y=0.5]{v_1}
\SO[Math, Lpos=180](v_1){v_2}
\EA[Math](v_1){v_4}
\SO[Math](v_4){v_3}
\Edge(v_1)(v_2)
\Edge(v_3)(v_4)
\end{tikzpicture}
&
$
\begin{array}{c}
\xrightarrow{(v_1,v_2)(v_4,v_3)} \\
\xleftarrow[(v_2,v_4)(v_3,v_1)]{}
\end{array}
$
&
\begin{tikzpicture}[baseline=0]
%[baseline=(current bounding box.center)]
\GraphInit[vstyle=Classic]
\SetGraphUnit{1}
\tikzset{VertexStyle/.append style = {minimum size = 2 mm}}
\Vertex[Math, Lpos=180, x=0, y=0.5]{v_1}
\SO[Math, Lpos=180](v_1){v_2}
\EA[Math](v_1){v_4}
\SO[Math](v_4){v_3}
\Edge(v_1)(v_3)
\Edge(v_2)(v_4)
\end{tikzpicture}
\\
\hline
\begin{tikzpicture}[baseline=0]
%[baseline=(current bounding box.center)]
\GraphInit[vstyle=Classic]
\SetGraphUnit{1}
\tikzset{VertexStyle/.append style = {minimum size = 2 mm}}
\Vertex[Math, Lpos=90, x=0, y=0]{v_1}
\Vertex[Math, x=1, y=0.5]{v_2}
\Vertex[Math, x=1, y=-0.5]{v_3}
\Loop[dist=1cm,dir=WE,style={-}](v_1)
\Edge(v_2)(v_3)
\end{tikzpicture}
&
$
\begin{array}{c}
\xrightarrow{(v_1,v_1)(v_2,v_3)} \\
\xleftarrow[(v_1,v_2)(v_3,v_1)]{}
\end{array}
$
&
\begin{tikzpicture}[baseline=0]
%[baseline=(current bounding box.center)]
\GraphInit[vstyle=Classic]
\SetGraphUnit{1}
\tikzset{VertexStyle/.append style = {minimum size = 2 mm}}
\Vertex[Math, Lpos=90, x=0, y=0]{v_1}
\Vertex[Math, x=1, y=0.5]{v_2}
\Vertex[Math, x=1, y=-0.5]{v_3}
\Edges(v_2,v_1,v_3)
\end{tikzpicture}
\\
\hline
\begin{tikzpicture}[baseline=0]
%[baseline=(current bounding box.center)]
\GraphInit[vstyle=Classic]
\SetGraphUnit{1}
\tikzset{VertexStyle/.append style = {minimum size = 2 mm}}
\Vertex[Math, Lpos=90, x=0, y=0]{v_1}
\Vertex[Math, Lpos=90, x=1, y=0]{v_2}
\Loop[dist=1cm,dir=WE,style={-}](v_1)
\Loop[dist=1cm,dir=EA,style={-}](v_2)
\end{tikzpicture}
&
$
\begin{array}{c}
\xrightarrow{(v_1,v_1)(v_2,v_2)} \\
\xleftarrow[(v_1,v_2)(v_2,v_1)]{}
\end{array}
$
&
\begin{tikzpicture}[baseline=0]
%[baseline=(current bounding box.center)]
\GraphInit[vstyle=Classic]
\SetGraphUnit{1}
\tikzset{VertexStyle/.append style = {minimum size = 2 mm}}
\Vertex[Math, Lpos=90, x=0, y=0]{v_1}
\Vertex[Math, Lpos=90, x=1, y=0]{v_2}
\Edge[style={bend left}](v_1)(v_2)
\Edge[style={bend right}](v_1)(v_2)
\end{tikzpicture}
\end{tabular}
\end{center}
\caption{Double edge swaps.
Note that if the swapped edges are not incident and none of them
is a loop, they can be swapped in two different ways. We
have omitted the double edge swaps
of the form $(v_1,v_1)(v_1,v_2)$ and $(v,v)(v,v)$
since they do not change the graph at all.
}
\label{fig:swaps}
\end{figure}

A weakly decreasing sequence is
said to be \emph{graphical} if it is the degree sequence of some
simple graph. The following theorem characterizes those sequences.
\begin{theorem}[Erd\H{o}s-Gallai \cite{ErdosGallai1960}]
A sequence of nonnegative integers $d_1 \ge d_2 \ge \dotsb \ge d_n$
is graphical if and only if $d_1+d_2+\dotsb+d_n$ is even and
\begin{equation}\label{eq:erdosgallai}
\sum_{i=1}^{k} d_i \le k(k-1) + \sum_{i=k+1}^n \min(d_i,k)
\end{equation}
holds for each $1\le k\le n$.
\end{theorem}
We will need this theorem later on, but only the ``only if'' part, and
since its proof is
a simple double counting argument we include it here for completeness:
Consider a simple graph with vertices $v_1,v_2\dotsc,v_n$ with degrees
$d_1\ge d_2\dotsb\ge d_n$, respectively.
The left-hand side of~\eqref{eq:erdosgallai}
gives the number of edge-vertex adjacencies among $v_1,v_2,\dotsc,v_k$.
The edge of each such adjacency
must have either one or two endpoints among
$v_1,v_2,\dotsc,v_k$;
the $k(k-1)$ term on the right-hand side gives the maximum possible
number of edge-vertex adjacencies in which both endpoints are
among $v_1,v_2,\dotsc,v_k$, and the remaining term on the right-hand
side upper-bounds the number of edges that have exactly one such endpoint.

\section{Results\label{sec:results}}
Our main result is the following.
\begin{theorem}\label{th:main}
Any loop-free multigraph whose degree sequence is graphical
can be transformed into a simple graph by a finite sequence of
admissible double edge swaps.
\end{theorem}
\noindent
Figure~\ref{fig:mainexample} shows an example.

\begin{figure}
\begin{center}
\begin{tabular}{ccccc}
\hspace{-7mm}
\begin{tikzpicture}[baseline=0]
%[baseline=(current bounding box.center)]
\GraphInit[vstyle=Classic]
\SetGraphUnit{1}
\tikzset{VertexStyle/.append style = {minimum size = 2 mm}}
\Vertices[Math]{circle}{v_1,v_2,v_3,v_4,v_5}
\Edges(v_5,v_1,v_2,v_4,v_5,v_3,v_2)
\Edge[style={bend left, out=20, in=160}](v_3)(v_4)
\Edge[style={bend left, out=20, in=160}](v_4)(v_3)
\end{tikzpicture}
&
\hspace{-3mm}$\xrightarrow{(v_3,v_4)(v_5,v_1)}$\hspace{-7mm}
&
\begin{tikzpicture}[baseline=0]
\GraphInit[vstyle=Classic]
\SetGraphUnit{1}
\tikzset{VertexStyle/.append style = {minimum size = 2 mm}}
\Vertices[Math]{circle}{v_1,v_2,v_3,v_4,v_5}
\Edges(v_5,v_3,v_4,v_2,v_3,v_1,v_2)
\Edge[style={bend left, out=20, in=160}](v_4)(v_5)
\Edge[style={bend left, out=20, in=160}](v_5)(v_4)
\end{tikzpicture}
&
\hspace{-3mm}$\xrightarrow{(v_2,v_1)(v_4,v_5)}$\hspace{-7mm}
&
\begin{tikzpicture}[baseline=0]
\GraphInit[vstyle=Classic]
\SetGraphUnit{1}
\tikzset{VertexStyle/.append style = {minimum size = 2 mm}}
\Vertices[Math]{circle}{v_1,v_2,v_3,v_4,v_5}
\Edges(v_5,v_4,v_3,v_2,v_5,v_3,v_1,v_4,v_2)
\end{tikzpicture}
\end{tabular}
\end{center}
\caption{A loop-free multigraph made simple by two admissible
double edge swaps. Note that in the original graph there is no admissible double edge swap that does not create a new multiple edge.}
\label{fig:mainexample}
\end{figure}
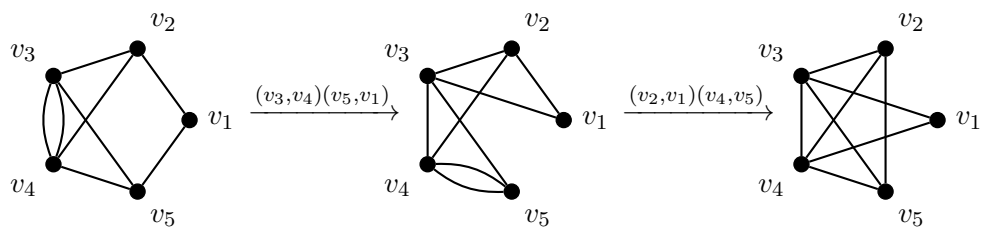

Let us state a simple consequence of Theorem~\ref{th:main}:
\begin{theorem}\label{th:loopy}
Any loopy multigraph whose degree sequence is graphical
can be transformed into a simple graph by a finite sequence of
admissible double edge swaps.
\end{theorem}
\begin{proof}
Consider a loopy multigraph whose degree sequence is graphical.
If there is a loop at some vertex $v$, since the degree sequence is graphical
there must be at least
one edge $\{v_1,v_2\}$ not incident to $v$, and the double
edge swap $(v,v)(v_1,v_2)$ is admissible and reduces the number
of loops. After removing all loops this way, the resulting loop-free
graph can be transformed into a simple graph by Theorem~\ref{th:main}.
\end{proof}

As mentioned in the introduction, Janson \cite[Remark 3.4]{Janson2019} conjectured a stronger
version of Question~\ref{qu:janson}, perhaps best phrased in terms
of a combinatorial game.

The \emph{loopy multigraph game} is played by the Angel and the Devil
as follows.
The starting position is a loopy multigraph $G$ with a graphical degree sequence.
In each move, the Devil chooses any loop or multiple edge
$e$ and then the Angel chooses any edge $e'$
and performs a double edge swap on $e$ and $e'$. The Angel
wins if she reaches a simple graph, and the Devil wins if the game
goes on forever.
\begin{conjecture}[Janson 2018]\label{co:loopy}
In the loopy multigraph game, the Angel has a winning strategy for any
starting position.
\end{conjecture}

We will prove this conjecture by showing that the Angel has a winning strategy
even if we change the rules of the game to make it much harder for her.
Let $H$ be a loop-free multigraph on a vertex set $V$.
In the \emph{loopy multigraph game with target $H$}, starting from a loopy multigraph $G$ on $V$ such that every vertex has the same degree in $G$ as in $H$,
in each move the Devil chooses any edge $e$ in $G$ such that $G$ has more
edges than $H$ of the same type as $e$ (that is, with the same endpoints),
and then the Angel chooses any edge $e'$ in $G$ not incident to $e$
and performs a double edge swap on $e$ and $e'$ in $G$. The Angel
wins if she reaches $H$, and the Devil wins if the game
goes on forever or if the Angel cannot make a move.

\begin{theorem}\label{th:loopytargetgame}
In the loopy multigraph game with target $H$,
the Angel has a winning strategy for any starting position.
\end{theorem}

It is not hard to see that Conjecture~\ref{co:loopy} follows from Theorem~\ref{th:loopytargetgame}: Given a loopy multigraph
$G$ with a graphical degree sequence, there is a simple (and therefore loop-free) graph $H$ with the
same degree sequence. In the loopy multigraph game, a valid move for
the Devil is to choose any edge which is a loop or whose multiplicity is larger than one.
Any such move is also valid in the loopy multigraph game with target $H$, since
the multiplicity of the same edge in $H$ is at most one. Thus, the Angel's winning
strategy in the target game is a winning strategy in the loopy multigraph game as well.

In the next section we will prove Theorem~\ref{th:loopytargetgame}
from which Conjecture~\ref{co:loopy} follows by the reasoning above.
Clearly, Theorem~\ref{th:main} follows from Theorem~\ref{th:loopytargetgame}
since we can choose the target $H$ to be any simple graph with the same degree sequence
as the original graph. However, in Section~\ref{sec:altproof} we will present an alternative
proof of Theorem~\ref{th:main} that is more elementary in the sense
that it does not rely on a choice of a simple graph as a target but
depends only on the ''only if'' part of the Erd\H{o}s-Gallai
theorem.

\section{Proof of Theorem~\ref{th:loopytargetgame}\label{sec:shortproof}}
\begin{proof}
Let $G$ be the current position in the game.
For each unordered pair of vertices $\{u,v\}$, let $m_G(\{u,v\})$ and
$m_H(\{u,v\})$ be the multiplicities of $\{u,v\}$ in $G$ and $H$, respectively.
If $m_G(\{u,v\}) > m_H(\{u,v\})$, choose $m_G(\{u,v\}) - m_H(\{u,v\})$ of the
edges of type $\{u,v\}$ in $G$ and call them \emph{exceeding}.
Analogously, if $m_H(\{u,v\}) > m_G(\{u,v\})$, choose $m_H(\{u,v\}) - m_G(\{u,v\})$ of the edges of type $\{u,v\}$ in $H$ and call them \emph{exceeding}.
All loops in $G$ are said to be exceeding as well.
Define the \emph{distance to $H$} to be the total number of exceeding edges.

Suppose the devil chooses an edge $e$. We may assume that $e$ is exceeding
since there are more edges in $G$ than in $H$ of the same type as $e$.

Colour some of the exceeding edges in $G$ blue and some of the exceeding edges in $H$ red such that
\begin{enumerate}[label=\upshape{\Roman*}]
\item\label{it:eblue}
$e$ is blue,
\item\label{it:balanced}
for each vertex, equally many blue as red edges are incident to it,
\item\label{it:minimal}
among all colourings with the above properties, we choose one where
the number of coloured edges is minimal.
\end{enumerate}
Clearly, there is at least one such colouring since colouring
all exceeding edges in $G$ blue and all exceeding edges in $H$ red satisfies conditions
\ref{it:eblue} and \ref{it:balanced}.

Let $u_1$ and $v_1$ be the endpoints of $e$. Since $e$ is blue, by
property \ref{it:balanced} there must be a red edge from $v_1$ to some
vertex $v_2$ and a blue edge from $v_2$ to some vertex $v_3$. Analogously,
there must be a red edge from $u_1$ to some vertex $u_2$ and a blue edge
from $u_2$ to some vertex $u_3$.

If $u_1=v_3$ and $v_1=u_3$, the red-blue alternating
circuit $v_1,v_2,u_1,u_2,v_1$ can be uncoloured without violating
conditions \ref{it:eblue} and \ref{it:balanced}, which is impossible
by the minimality condition \ref{it:minimal}. Thus, at least one of the
inequalities $u_1\ne v_3$ and $v_1\ne u_3$ holds. For symmetry reasons,
we may assume that $u_1\ne v_3$. Since no two vertices can have both a blue and red edge between them, we have $u_1\ne v_2$ and $v_1\ne v_3$, and since
$H$ is loop-free, we have $v_1\ne v_2$. These four inequalities show that
the edges $e=(u_1,v_1)$ and $(v_2,v_3)$ are not incident, so
the Angel can perform the double edge swap $(u_1,v_1)(v_2,v_3)$.
This removes a blue exceeding edge of type $\{u_1,v_1\}$, a blue exceeding
edge of type $\{v_2,v_3\}$ and at least one red exceeding edge of type
$\{v_1,v_2\}$ while introducing at most one new blue exceeding edge of type
$\{v_3,u_1\}$, so the distance to $H$ decreases. (Note that this holds also for the case where $u_1=v_1$ and $v_2=v_3$.)

Thus, the Angel can always decrease the distance to $H$. Eventually
the distance will be zero and $H$ is reached.
\end{proof}

\section{Alternative proof of Theorem~\ref{th:main}\label{sec:altproof}}
Note that Theorem~\ref{th:main} follows from Theorem~\ref{th:loopytargetgame}
since we can choose the target $H$ to be any simple graph with the same degree sequence as the original graph.
In this section we present an alternative proof of Theorem~\ref{th:main}
which does not rely on such a choice. In fact, the only way it exploits
the graphicality of the degree sequence is via the inequalities
guaranteed by the easy ``only if'' direction of the Erd\H{o}s-Gallai theorem.
With this in mind, the harder ``if'' direction of the Erd\H{o}s-Gallai theorem follows easily from Theorem~\ref{th:loopy} since any sequence of nonnegative integers with an even sum clearly is the degree sequence of some loopy multigraph.

\subsection{Ordering vertices, edges and graphs}
Fix a vertex set $V$ with prescribed degrees given by a function
$d:V\rightarrow\mathbb{N}$.
Choose a total order on $V$ with the property that
$d(u)<d(v)$ implies $u<v$. This induces a total order
on unordered pairs of vertices defined by, for $u_1>u_2$ and $v_1>v_2$, letting
$\{u_1,u_2\}\le\{v_1,v_2\}$ if either $u_1<v_1$ or
$u_1=v_1$ and $u_2\le v_2$. This in turn induces a (strict) partial order
on loop-free multigraphs with vertex set $V$ and the prescribed
degrees $\deg v=d(v)\ \forall v\in V$, by letting $G<H$ if
$H$ is not simple and its maximal non-simple edge is
larger than all non-simple edges in $G$, or the
maximal non-simple edges of $G$ and $H$ are equal
but its multiplicity is strictly larger in $H$ than in $G$.

\begin{proposition}\label{pr:main}
Any non-simple loop-free multigraph whose degree sequence is graphical
can be transformed into a smaller graph by a finite sequence of
admissible double edge swaps.
\end{proposition}

Before we prove the proposition we present a technical device
that can sometimes reduce the multiplicity of an edge without adding
too many new multiple edges.

\subsection{A swapping lemma}
\begin{lemma}\label{lm:cycle}
Let $m\ge2$ be an integer and
let $v_1,v_2,\dotsc,v_{2m}$ be a sequence of vertices
in a loop-free multigraph.
Suppose the following holds.
\begin{enumerate}
\item[(a)]
$v_1$ is distinct from all of $v_2,v_3,\dotsc,v_{2m}$.
\item[(b)]
$v_i\ne v_{i+1}$ for $i=1,2,\dotsc,2m-1$, and the corresponding
unordered pairs $\{v_i,v_{i+1}\}$ are all distinct from each other and
from $\{v_1,v_{2m}\}$.
\item[(c)]
There are edges of type
$\{v_{2j},v_{2j+1}\}$ for $j=1,2,\dotsc,m-1$.
\item[(d)]
There are multiple edges of type $\{v_1,v_{2m}\}$.
\end{enumerate}
Then there is a sequence of admissible double edge swaps that reduces the
multiplicity of $\{v_1,v_{2m}\}$ by one without
adding any new non-simple edge except possibly those
edges of types $\{v_{2j-1},v_{2j}\}$, $j=1,2,\dotsc,m$ that were already present.
\end{lemma}
\begin{proof}
Since the unordered pairs $\{v_{2m-2},v_{2m-1}\}$
and $\{v_{2m-1},v_{2m}\}$ are distinct,
the vertices $v_{2m-2}$, $v_{2m-1}$ and $v_{2m}$ are all distinct, and we can
perform the admissible swap $(v_1,v_{2m})(v_{2m-1},v_{2m-2})$,
see Figure~\ref{fig:inductionstep}.
That reduces
the multiplicity of $\{v_1,v_{2m}\}$ by one and introduces
only two new edges, of type $\{v_1,v_{2m-2}\}$ and $\{v_{2m-1},v_{2m}\}$.
We are done unless $\{v_1,v_{2m-2}\}$ is now neither simple nor equal 
to any $\{v_{2j-1},v_{2j}\}$, $j=1,2,\dotsc,m$. In that case,
$v_{2m-2}\ne v_2$ and $m \ge 3$, so
the sequence $v_1,v_2,\dotsc,v_{2(m-1)}$ satisfies properties
(a) to (d). Then, by induction, there is a sequence of admissible double edge swaps
that reduces the multiplicity of $\{v_1,v_{2(m-1)}\}$ to its original value without
adding any new non-simple edge except possibly those edges of
type $\{v_{2j-1},v_{2j}\}$, $j=1,\dotsc,m-1$ that were already present.
\end{proof}

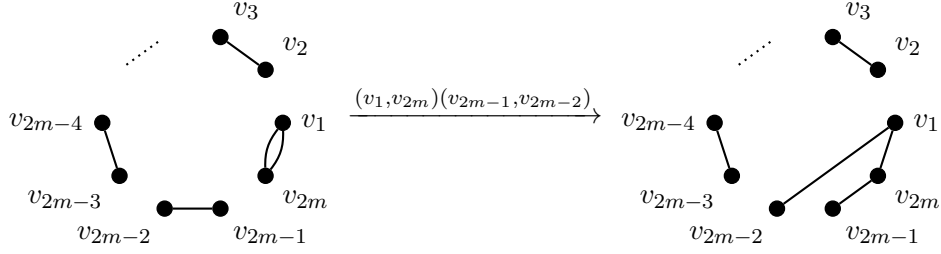
\begin{figure}
\begin{tikzpicture}[baseline=0]
%[baseline=(current bounding box.center)]
\GraphInit[vstyle=Classic]
\SetGraphUnit{1.2}
\tikzset{VertexStyle/.append style = {minimum size = 2 mm}}
\Vertex[Math, d=1.2, a=0]{v_1}
\Vertex[Math, d=1.2, a=36, Lpos=36]{v_2}
\Vertex[Math, d=1.2, a=72, Lpos=72]{v_3}
\tikzset{VertexStyle/.append style = {color=white}}
\Vertex[Math, d=1.2, a=108, Lpos=108, NoLabel]{v_4}
\Vertex[Math, d=1.2, a=144, Lpos=144, NoLabel]{v_5}
\tikzset{VertexStyle/.append style = {color=black}}
\Vertex[Math, d=1.2, a=180, Lpos=180]{v_{2m-4}}
\Vertex[Math, d=1.2, a=-144, Lpos=-144]{v_{2m-3}}
\Vertex[Math, d=1.2, a=-108, Lpos=-108]{v_{2m-2}}
\Vertex[Math, d=1.2, a=-72, Lpos=-72]{v_{2m-1}}
\Vertex[Math, d=1.2, a=-36, Lpos=-36]{v_{2m}}
\Edge(v_2)(v_3)
\Edge[style={dotted}](v_4)(v_5)
\Edge(v_{2m-4})(v_{2m-3})
\Edge(v_{2m-2})(v_{2m-1})
\Edge[style={bend left, out=20, in=160}](v_1)(v_{2m})
\Edge[style={bend left, out=20, in=160}](v_{2m})(v_1)
\end{tikzpicture}
$\xrightarrow{(v_1,v_{2m})(v_{2m-1},v_{2m-2})}$
\begin{tikzpicture}[baseline=0]
%[baseline=(current bounding box.center)]
\GraphInit[vstyle=Classic]
\SetGraphUnit{1.2}
\tikzset{VertexStyle/.append style = {minimum size = 2 mm}}
\Vertex[Math, d=1.2, a=0]{v_1}
\Vertex[Math, d=1.2, a=36, Lpos=36]{v_2}
\Vertex[Math, d=1.2, a=72, Lpos=72]{v_3}
\tikzset{VertexStyle/.append style = {color=white}}
\Vertex[Math, d=1.2, a=108, Lpos=108, NoLabel]{v_4}
\Vertex[Math, d=1.2, a=144, Lpos=144, NoLabel]{v_5}
\tikzset{VertexStyle/.append style = {color=black}}
\Vertex[Math, d=1.2, a=180, Lpos=180]{v_{2m-4}}
\Vertex[Math, d=1.2, a=-144, Lpos=-144]{v_{2m-3}}
\Vertex[Math, d=1.2, a=-108, Lpos=-108]{v_{2m-2}}
\Vertex[Math, d=1.2, a=-72, Lpos=-72]{v_{2m-1}}
\Vertex[Math, d=1.2, a=-36, Lpos=-36]{v_{2m}}
%\Vertices[Math]{circle}{v_1,v_2,v_3,v_4,v_5,v_6,v_7,v_8}
\Edge(v_2)(v_3)
\Edge[style={dotted}](v_4)(v_5)
\Edge(v_{2m-4})(v_{2m-3})
\Edge(v_{2m-2})(v_1)
\Edges(v_{2m-1},v_{2m},v_1)
\end{tikzpicture}
\caption{The induction step in the proof of Lemma~\ref{lm:cycle}.}
\label{fig:inductionstep}
\end{figure}

\subsection{Proof of Proposition~\ref{pr:main}}
Let $G$ be a non-simple loop-free multigraph that cannot be
transformed to a smaller graph by a finite number of admissible
double edge swaps. To prove the proposition, we must show that
the degree sequence of $G$ is not graphical.

To this end we will need a bunch of lemmas, and they are
all implicitly referring to $G$ and to the following notation and terminology.

Let $\{u_1,u_2\}$ with $u_1>u_2$ be the
maximal non-simple edge in $G$. The vertices
other than $u_1$ and $u_2$ will be called \emph{ordinary} vertices.
For $i=1,2$, let $V_i$ and $\overline{V_i}$ be the sets of ordinary vertices
that have an edge to
$u_i$ and that do not have an edge to $u_i$, respectively.
An ordinary vertex is called \emph{small} if it is smaller than $u_1$
and \emph{large} if it is larger than $u_1$.

\begin{lemma}\label{lm:one}
There is no edge between a vertex $v_1$ in $\overline{V_1}$
and a vertex $v_2$ that is small or belongs to $\overline{V_2}$.
\end{lemma}
\begin{proof}
If there was such an edge $\{v_1,v_2\}$ the
admissible double edge swap $(u_1,u_2)(v_2,v_1)$ would
reduce the multiplicity of $\{u_1,u_2\}$ without creating any
new non-simple edge, except possibly for $\{u_2,v_2\}$ if
$v_2$ is small, and then $\{u_2,v_2\}$ is smaller than $\{u_1,u_2\}$.
This contradicts the assumptions on $G$.
\end{proof}

\begin{lemma}\label{lm:two}
All large vertices belong to $V_1\cap V_2$, and every large vertex is adjacent
to some vertex in $\overline{V_1}$ and to some vertex in $\overline{V_2}$.
\end{lemma}
\begin{proof}
Let $v$ be any large vertex.
By the maximality of the non-simple edge $\{u_1,u_2\}$
all edges from $v$ are simple,
so the degree of $v$ equals the number of vertices adjacent to $v$.
Since $\deg v\ge\deg u_1$ and $u_1$ has multiple edges to $u_2$,
$v$ is adjacent to some vertex $v_1$ in $\overline{V_1}$.
Lemma~\ref{lm:one} now yields that $v$ belongs to $V_2$.
Analogously,
since $\deg v\ge\deg u_2$ and $u_2$ has multiple edges to $u_1$,
$v$ is adjacent to some vertex $v_2$ in $\overline{V_2}$, and,
by Lemma~\ref{lm:one},
$v$ belongs to $V_1$.
\end{proof}

\begin{lemma}\label{lm:four}
Any ordinary vertex adjacent to a small vertex must be adjacent
to all large vertices, except for itself (if it is large).
\end{lemma}
\begin{proof}
Suppose an ordinary vertex $v_1$ is adjacent to a small vertex $v_2$ but not
to some large vertex $v_3\ne v_1$. By Lemma~\ref{lm:two}, $v_3$ is adjacent
to some vertex $v_4$ in $\overline{V_1}$. (Note that $v_4$ might be identical
to $v_2$.) Applying Lemma~\ref{lm:cycle}
to the sequence $u_1,v_4,v_3,v_1,v_2,u_2$ shows that there is a
sequence of admissible double edge swaps that reduces the multiplicity
of $\{u_1,u_2\}$ without adding any new non-simple edge except possibly
those edges among $\{u_1,v_4\}$, $\{v_3,v_1\}$ and $\{v_2,u_2\}$
that were already present. By construction, $\{u_1,v_4\}$ and $\{v_3,v_1\}$
were not present, so the only possible new non-simple edge is
$\{v_2,u_2\}$, which is smaller than $\{u_1,u_2\}$ since $v_2$ is small.
This contradicts the assumptions on $G$, and we conclude that the
first sentence in the lemma holds.
\end{proof}

\begin{lemma}\label{lm:largeadjacent}
All vertices in $L\cup\{u_1,u_2\}$ are adjacent.
\end{lemma}
\begin{proof}
By Lemma~\ref{lm:two}, $u_1$ and $u_2$ are adjacent to all large vertices.
Again by Lemma~\ref{lm:two}, any large vertex is adjacent to some small
vertex, since all vertices in $\overline{V_1}$ are small. Now it follows
from Lemma~\ref{lm:four} that any large vertex is adjacent to all
other large vertices.
Finally, $u_1$ and $u_2$ are adjacent by definition.
\end{proof}

\begin{lemma}\label{lm:seven}
A small vertex not adjacent to any small vertex must be smaller than $u_2$.
\end{lemma}
\begin{proof}
Suppose there is a small vertex $v>u_2$ not adjacent to any small vertex.
By the maximality of the non-simple edge $\{u_1,u_2\}$,
all edges from $v$ to any vertex greater than
or equal to $u_1$ are simple. Thus, the degree of $v$ is at most $\ell+1+m$,
where $\ell$ is the number of large vertices and $m$ is the
multiplicity of the edge $\{v,u_2\}$ (possibly zero).
By Lemma~\ref{lm:two}, $u_2$ is adjacent to all large vertices,
so its degree is at least $\ell+m+2$. This shows that
$\deg u_2 > \deg v$, which contradicts the assumption
that $v>u_2$.
\end{proof}

\begin{lemma}\label{lm:eight}
If there is an edge between small vertices somewhere in the graph,
then every small vertex in $V_2$ is adjacent to some small vertex.
\end{lemma}
\begin{proof}
Suppose there are small adjacent vertices $v_1$ and $v_2$
and a small vertex $v$ in $V_2$ not adjacent to any small vertex.
By Lemma~\ref{lm:one}, $v_2$ is adjacent to $u_1$.

Applying Lemma~\ref{lm:cycle} on the sequence $u_1,v,u_2,v_1,v_2,u_2$
shows that there is a sequence of admissible double edge swaps
that reduces the multiplicity of $\{u_1,u_2\}$ and adds
no new non-simple edge except possibly
$\{u_1,v\}$, $\{u_2,v_1\}$ and $\{v_2,u_2\}$.
But all these edges are smaller than $\{u_1,u_2\}$,
the first one since $v<u_2$ by Lemma~\ref{lm:seven}.
This contradicts the assumptions on $G$.
\end{proof}

Let $L$ be the set of large vertices and let $\ell$ be the number of them.
\begin{lemma}\label{lm:nine}
If there is an edge between small vertices somewhere in the graph,
then every small vertex $v$ has at least $\min(\ell+1,\deg v)$
edges to vertices in $L\cup\{u_1\}$.
\end{lemma}
\begin{proof}
Suppose there is an edge between small vertices somewhere in the graph,
and consider a small vertex $v$.

First suppose $v$ is adjacent to a small vertex.
Then by Lemma~\ref{lm:four} it is
adjacent to all large vertices, and by Lemma~\ref{lm:one}
it belongs to $V_1$, so
it is adjacent to all $\ell+1$ vertices in $L\cup\{u_1\}$.

Now, suppose instead that $v$ is not adjacent to any small vertex.
Then, by Lemma~\ref{lm:eight} it does not belong to $V_2$,
and clearly the degree of $v$ equals the number of edges from $v$
to $L\cup\{u_1\}$.
\end{proof}

\begin{lemma}\label{lm:final}
The degree sequence of $G$ is not graphical.
\end{lemma}
\begin{proof}
We treat two cases separately.

%\begin{enumerate}
{\bf Case 1:} No two small vertices are adjacent.

By Lemma~\ref{lm:seven} all small vertices are smaller than $u_2$.
Also, by Lemma~\ref{lm:largeadjacent} all vertices in $L\cup\{u_1,u_2\}$ are adjacent, and $\{u_1,u_2\}$ is non-simple.
It follows that
\[
\sum_{v\ge u_2} \deg v \ge (\ell+2)(\ell+1)+2 + m,
\]
where $m$ is the number of edges (counted with multiplicity) between a vertex in
$L\cup\{u_1,u_2\}$ and a small vertex.
Since no two small vertices are adjacent, we have
$\sum_{v<u_2} \deg v = m$. Plugging this into the inequality above
yields
\[
\sum_{v\ge u_2} \deg v \ge (\ell+2)(\ell+1) + 2 + \sum_{v<u_2} \deg v.
\]
Note that,
since all small vertices are smaller than $u_2$, there are
exactly $\ell+2$ vertices larger than or equal to $u_2$,
namely $u_1$, $u_2$ and all the large vertices.
Letting $d_i$ be the degree of the $i$-th largest vertex, the
above inequality can be written
\[
\sum_{i=1}^{\ell+2} d_i \ge (\ell+2)(\ell+1) + 2 + \sum_{i=\ell+3}^n d_i,
\]
where $n$ is the number of vertices. Now it follows from the Erd\H{o}s-Gallai
theorem (with $k=\ell+2$)
that the degree sequence of $G$ is not graphical.

{\bf Case 2:} There are at least two adjacent small vertices.

As before, by Lemma~\ref{lm:largeadjacent} all vertices in $L\cup\{u_1,u_2\}$
are adjacent, and $\{u_1,u_2\}$ is non-simple. It follows that
the number of edges (counted with multiplicity) between a vertex
in $L\cup\{u_1\}$ and a vertex in $L\cup\{u_1,u_2\}$ is
at least $(\ell+1)^2+1$, so
\begin{equation}\label{eq:lhsineq}
\sum_{v\ge u_1} \deg v \ge (\ell+1)^2 + 1 + m',
\end{equation}
where $m'$ is the number of edges between a vertex in
$L\cup\{u_1\}$ and a small vertex. By Lemma~\ref{lm:nine},
\[
m'\ge \sum_{\text{small\ }v} \min(\ell+1,\deg v),
\]
and thus
\[
\sum_{v<u_1} \min(\ell+1,\deg v)\le \ell+1+m'.
\]
Combining this with~\eqref{eq:lhsineq} gives
\[
\sum_{v\ge u_1} \deg v \ge (\ell+1)\ell + 1 + \sum_{v<u_1} \min(\ell+1,\deg v),
\]
and it follows from the Erd\H{o}s-Gallai theorem (with $k=\ell+1$)
that the degree sequence of $G$ is not graphical.
%\end{enumerate}
\end{proof}

Proposition~\ref{pr:main} follows from Lemma~\ref{lm:final},
and Theorem~\ref{th:main} then follows from the proposition.

\section{Open questions\label{sec:open}}
In this paper, we have focused on the \emph{existence} of a sequence of
admissible double edge swaps that makes a graph simple. For further research, it seems natural to ask about the \emph{length} of such a sequence.
\begin{question}\label{qu:open}
What is the minimum number of admissible double edge swaps needed to transform a given loopy multigraph $G$ with a graphical degree sequence into a simple graph?
\end{question}
As we saw in Figure~\ref{fig:mainexample}, it is not always possible
to decrease the number of multiple edges by a double edge swap. On the other hand,
some double edge swaps might decrease the number of multiple edges by two.

A related question is whether the word ``admissible'' in Question~\ref{qu:open}
matters. Are there situations where the fastest road to simplicity require double edge swaps with only simple edges?

\section{Acknowledgement}
I want to thank Svante Janson for posing the question that initiated this research and for pointing out errors in an earlier version of the paper. I am also very grateful to the anonymous referees, one of which suggested an approach that led to the proof of Conjecture~\ref{co:loopy}.

This work was mostly conducted while the author was a researcher as Uppsala University. It was supported by the Knut and Alice Wallenberg Foundation.

\bibliographystyle{plain}

%\bibliography{double-edge-swaps}
\end{document}